\documentclass[11pt,oneside]{amsart}
\usepackage{amssymb, amsmath, amsthm}

\usepackage{epsfig}
\usepackage{graphicx}
\usepackage{color}
\usepackage{enumerate}
\usepackage{amsrefs}
\usepackage[pagewise]{lineno}

\usepackage{pinlabel}

\theoremstyle{plain}
\newtheorem{theorem}{Theorem}[section]

\newtheorem*{theorem*}{Theorem}

\newtheorem{corollary}[theorem]{Corollary}

\theoremstyle{definition}

\theoremstyle{definition}

\newcommand{\up}{\uparrow}

\newcommand{\nil}{\varnothing}

\newcommand{\wihat}{\widehat}
\newcommand{\defn}[1]{\emph{#1}}

\newcommand{\boundary}{\partial}

\newcommand{\mc}[1]{\mathcal{#1}}

\newcommand{\genus}{\operatorname{g}} 
\newcommand{\inter}[1]{\mathring{#1}}
\newcommand{\down}{\downarrow}

\newcommand{\co}{\mskip0.5mu\colon\thinspace}

\begin{document}

   \title[]{Genus bounds bridge number for high distance knots}
   \author{Ryan Blair}
   \author{ Marion Campisi}
   \author{ Jesse Johnson}
   \author{ Scott A. Taylor}
   \author{ Maggy Tomova}
   \email{}
   \thanks{The authors are grateful to the American Institute of Mathematics for its support through the SQuaREs program.}

\begin{abstract}
If a knot $K$ in a closed, orientable 3-manifold $M$ has a bridge surface $T$ with distance at least 3 in the curve complex of $T - K$, then the genus of any essential surface in its exterior with non-empty, non-meridional boundary gives rise to an upper bound for the bridge number of $K$ with respect to $T$. In particular, a nontrivial, aspherical, and atoroidal knot $K$ with such a bridge surface has its bridge number bounded by 5 if $K$ has a non-trivial reducing surgery; 6 if $K$ has a non-trivial toroidal surgery; and $4g + 2$ if $K$ is null-homologous and has Seifert genus $g$.
\end{abstract}
\maketitle
\date{\today}

\section{Introduction}
If a knot $K$ in a 3-manifold $M$ is in bridge position with respect to a Heegaard surface $T$ for $M$, both bridge number $b(T)$ and distance $d_\mc{C}(T)$ are integer measures of the complexity of the bridge position. Both give rise to knot invariants (since we can minimize over all possible bridge positions for $K$) and both reflect, to some degree, the topology and geometry of the knot exterior. Although, in general, there is no relationship between bridge number and the genus of essential surfaces in the knot exterior, we show that, for knots with bridge surfaces of distance at least 3, the bridge number is bounded above by an explicit linear function of the genus of such a surface, assuming the surface has non-empty, non-meridional boundary. As a consequence, we show:

\begin{theorem}\label{apps1}
Suppose that $K$ is a non-trivial knot in a closed, connected orientable 3-manifold $M$. Let $T$ be a bridge surface for $(M,K)$, other than a 2 or 4 punctured sphere, and with $d_\mc{C}(T) \geq 3$. Then the following hold:
\begin{enumerate}
\item If $K$ is null-homologous in $M$ then $b(T) \leq 4g(K) + 2$ where $g(K)$ is the minimum genus of a Seifert surface for $K$.
\item If the exterior of $K$ is aspherical and non-trivial Dehn surgery on $K$ produces a reducible 3-manifold, then $b(T) \leq 5$.
\item If the exterior of $K$ is atoroidal, and non-trivial Dehn surgery on $K$ produces a toroidal 3-manifold, then $b(T) \leq 6$. Furthermore, if the surgery slope is non-longitudinal, then $b(T) \leq 5$.
\end{enumerate}
\end{theorem}

The first conclusion is surprising, for if we drop the hypothesis that $d_\mc{C}(T) \geq 3$, there are genus 1 knots of arbitrarily high bridge number. For example, let $J_n$ be a sequence of knots in $S^3$ such that the minimum bridge number of a bridge sphere for $J_n$ goes to infinity with $n$. If $K_n$ is the Whitehead double of $J_n$, the Seifert genus of $K_n$ is 1 but the bridge number of $K_n$ is at least twice the bridge number of $J_n$ \cites{Schubert, Schultens}. We expect that there are genus 1 hyperbolic knots of arbitrarily large bridge number, but constructing them is beyond the scope of this paper.

The second conclusion puts strong restrictions on any potential counterexample to the cabling conjecture \cite{GAS}. For, suppose that a counterexample $K \subset S^3$ is in minimal bridge position with respect to a Heegaard sphere $T$. Hoffman \cite{H1} showed that $b(T) \geq 5$ and, in \cite{H2}, claims he has also proved (in unpublished notes) that $b(T) \geq 6$. If that result is correct, then our result reduces the cabling conjecture to studying knots having bridge spheres $T$ satisfying the simple combinatorial condition $d_\mc{C}(T) \leq 2$.  In a forthcoming paper, we will describe all knots in $S^3$ with a bridge sphere satisfying $d_\mc{C}(T) = 2$. Many of them, it turns out, contain an essential meridional planar surface in their exterior, much like when $d_\mc{C}(T) = 1$. Thus, resolving the cabling conjecture for knots with an essential tangle decomposition would be an important step towards resolving the cabling conjecture in general. Towards that end, Hayashi \cite{Hy} has shown that if $K$ has an essential tangle decomposition such that no two strands of either tangle are parallel, then $K$ satisfies the cabling conjecture and Taylor \cite{T} has shown that if $K$ is formed by attaching a ``complicated'' band to a two component link (e.g. if $K$ is a band sum) then $K$ also satisfies the cabling conjecture. However, the proof of the cabling conjecture for the case when $K$ has an essential meridional planar surface in its exterior remains incomplete.

With reference to the last conclusion, we note that if a knot $K \subset S^3$ lies in some complicated way on a knotted genus 2 surface $W \subset S^3$, then a Dehn surgery on $K$ corresponding to the (integral) slope of $W \cap (S^3 - \inter{\eta}(K))$ will likely produce a toroidal 3-manifold. Presumably, if the knotting of $W$ is complicated enough, then the bridge number of $K$ with respect to a Heegaard sphere can be made arbitrarily high. Eudave-Mu\~noz \cite{EM} has given examples of hyperbolic knots in $S^3$ with toroidal surgeries of half-integral slope. Our result shows that all knots with toroidal surgeries and with high bridge number cannot also have high distance bridge surfaces. 

\section{Background and Previous Results}\label{background}
Let $M$ be a compact, connected, orientable 3-manifold (possibly with boundary) and let $K \subset M$ be a nontrivial knot with a compact, orientable surface $S$ properly embedded in its exterior $M_K$. Let $\boundary_0 S = \boundary S \cap \boundary M$ and let $\boundary_K S = \boundary S - \boundary_0 S$. Assume that all the components of $\boundary_K S$ are parallel, essential and non-meridional curves. Let $\Delta$ be the minimal intersection number between a component of $\boundary_K S$ and a meridian of $K$. We say that a simple closed curve $\sigma \subset S$ is \defn{essential} in $S$ if it does not bound a disc in $S$ and if it is not isotopic to a component of $\boundary_K S$. Curves isotopic to $\boundary_0 S$ are considered to be essential for the purposes of this paper. An arc properly embedded in $S$ is \defn{essential} if it is not boundary parallel.

A \defn{compressionbody} $C$ is any space obtained from $F \times [0,1]$, with $F$ a closed connected surface, by attaching 2-handles and 3-handles along $F \times \{0\}$. We let $\boundary_+ C = F \times \{1\}$ and $\boundary_- C = \boundary C - \boundary_+ C$. The union $\tau$ of properly embedded arcs in $C$ is \defn{trivial} if $\tau$ is isotopic into $\boundary_+ C$ relative to $\boundary \tau$. If $\tau \subset C$ is trivial, a \defn{spine} $\Gamma$ for $(C,\tau)$ is an embedded graph in $C$ such that the exterior of $\Gamma \cup \boundary_- C$ is homeomorphic to $\boundary_+ C \times [0,1]$ intersecting $\tau$ in a union of vertical arcs.

A \defn{bridge surface} for $(M,K)$ is a closed separating surface $T \subset M$ such that the closure of each component of $M - T$ is a compressionbody intersecting $K$ in trivial arcs. We let $T_\up$ and $T_\down$ denote the closures of the components of $M - T$. A simple closed curve $\sigma \subset T_K = T - \inter{\eta}(K)$ is \defn{essential} if $\sigma$ does not bound a disc or once punctured disc in $T_K$. An arc $\sigma$ properly embedded in $T_K$ is \defn{essential} if it is not boundary parallel in $T_K$. The \defn{curve complex} $\mc{C}(T)$ of $T$ has vertices equal to isotopy classes of essential simple closed curves in $T_K$. Two vertices in $\mc{C}(T)$ are joined by an edge if the vertices have disjoint representatives in $T_K$. If $T$ is a surface other than a torus with 0 or 1 punctures or a sphere with 4 or fewer punctures, then $\mc{C}(T)$ is connected. Since, for us, $T$ is a bridge surface for a knot it cannot be a zero or once-punctured torus. The disc sets $\mc{D}^\up_\mc{C}$ and $\mc{D}^\down_\mc{C}$ for $\mc{C}(T)$ consist of those vertices that bound compressing discs for $T_K$ in $T_\up - \inter{\eta}(K)$ and $T_\down-\inter{\eta}(K)$ respectively. The \defn{bridge distance} $d_\mc{C}(T)$ of a bridge surface $T$ is defined to be the distance from $\mc{D}^\down_\mc{C}$ to $\mc{D}^\up_\mc{C}$ in $\mc{C}(T)$. This definition is a ready generalization of Hempel's definition \cite{Hempel} of distance for Heegaard surfaces (i.e. when $K = \nil$). If $d_\mc{C}(T) = 0$, there is a sphere in $M$ intersecting $T_K$ in a single essential loop. This implies that either $M - K$ is reducible or that $T$ is a stabilized bridge surface (in the sense of \cite{HS}.) If $d_\mc{C}(T) = 1$, the bridge surface can be untelescoped and, in most cases, there is an essential meridional surface in $M - K$ of genus at most the genus of $T$ \cite{HS}. The paper \cite{BTY} shows that bridge surfaces $T$ exist with $d_\mc{C}(T)$ arbitrarily high, and in \cite{IS} this result is improved to show that such surfaces continue to exist if the 3-manifold is fixed.

Rather than measuring the distance of $T$ in $\mc{C}(T)$, distance could be measured in the ``arc and curve complex''\cite{BS}. This gives rise to a different integer complexity of $T$, denoted $d_\mc{AC}(T)$. It is always the case that $d_\mc{AC}(T) \leq d_\mc{C}(T) \leq 2d_\mc{AC}(T)$ \cite[Lemma 2.9]{MHL1}. 

In \cite{MHL1}, we proved that there is a relationship between the distance $d_\mc{AC}(T)$ of a bridge surface $T$ for a knot $K$ in a compact, connected, orientable 3-manifold $M$ and the genus $\genus(S)$ of either an essential surface or a Heegaard surface $S$ in a manifold obtained by performing non-trivial Dehn surgery on $K$. Among other results, we showed that if $K \subset S^3$ has a surgery producing a reducible or toroidal 3-manifold, then $d_\mc{AC}(T) \leq 2$. Theorem \ref{apps1} refines this result by showing that $d_\mc{C}(T) \leq 2$ when $b(T)$ is large enough. We do not need to consider $d_\mc{AC}$ in this paper. 

\section{Theorems and Proofs}

Theorem \ref{apps1} is a specialization of:

\begin{corollary}\label{apps2}
Let $K$ be a nontrivial knot in a closed, connected, orientable 3-manifold $M$ and let $T$ be a bridge surface for $(M,K)$ with $\mc{C}(T)$ connected and $d_\mc{C}(T)\geq 3$.  Suppose that $M_K$ contains an essential, properly embedded, compact, connected orientable surface of genus $g$ with non-empty, non-meridional boundary on the boundary of a regular neighborhood of $K$. Then
\[ b(T)\leq \max(5,4g +2)\]
\end{corollary}

In fact, Theorem \ref{Main Theorem} shows that the conclusion holds even if we relax the requirement that $S$ is essential. Before stating the theorem, we establish some notation and definitions.

Let $\Gamma_\down$ and $\Gamma_\up$ be spines for $(T_\down, K \cap T_\down)$ and $(T_\up, K \cap T_\up)$ respectively. The complement of $\Gamma_\down \cup \Gamma_\up \cup \boundary M$ in $M$ is homeomorphic to $T \times (0,1)$. Let $h\co M \to [0,1]$ be projection onto the second factor and extend $h$ so that $h(\boundary_- T_\down \cup \Gamma_\down) = 0$ and $h(\boundary_- T_\up \cup \Gamma_\up) = 0$. (Without loss of generality, we may assume that the choice of labels $T_\down$ and $T_\up$ makes this extension continuous.) The map $h$ is called a \defn{sweepout} of $(M,K)$ by $T$. For all $t \in (0,1)$, the surface $T_t = h^{-1}(t)$ is a surface isotopic to $T$ and transverse to $K$. Perturb $h$ so that $h|_S$ is a Morse function with critical points having distinct critical values. By putting a flat metric on the frontier of $K$, and isotoping $S$ and $h$ so that $\boundary_K S$ and $\boundary (T_t \cap M_K)$ are the union of geodesics for all $t$, we may assume that the quantity $|\boundary S \cap \boundary (T_t \cap M_K)|$ is constant. Hence,
\[
|\boundary S \cap \boundary (T_t \cap M_K)| = 2b(T)|\boundary_K S|\Delta.
\]

An interval $[a,b] \subset [0,1]$ is \defn{essential} for $S$ relative to $h$ if $a$ and $b$ are regular values for $h|_S$ and if for all regular values $t\in [a,b]$ all components of $T_t \cap S$ are essential in both surfaces. Let $\epsilon > 0$ be less than half the minimum distance between adjacent critical points.  An essential interval $[a,b]$ is \defn{maximally essential} for $S$ if, for the critical value $a_-$ just below $a$ and the critical value $b_+$ just above $b$, some arc or circle $\alpha$ of $T_{a_- - \epsilon} \cap S$ is essential in $T$ but bounds a compressing or boundary compressing disc for $T$ that lies in $T_\down$ and some arc or circle $\beta$ of $T$ is essential in $T_{b_+ + \epsilon} \cap S$ but bounds a compressing or boundary compressing disc for $T$ that lies in $T_\up$.

\begin{theorem}[Main Theorem]\label{Main Theorem}
Assume that $\mc{C}(T)$ is connected, $d_\mc{C}(T) \geq 3$, and that there is a maximally essential interval for $S$ relative to $T$. Then,
\[
(b(T) - 4)\Delta \leq \frac{4\genus(S) - 4|S| +2 |\boundary_0 S|}{|\boundary_K S|} + 2.
\]
\end{theorem}

\begin{proof}
 The proof is a variation of \cite[Theorem 3.1]{MHL1}.  Let $t_-, t_+ \in (0,1)$ be regular values of $h|_S$ such that there is a unique critical value $v$ of $h|_S$ in $[t_-,t_+]$. As $t \in [t_-,t_+]$ passes through $v$, a band is attached to one or two components of $T_{t_-} \cap S$ to create one or two components of $T_{t_+} \cap S$. All components of $T_{t_-} \cap S$ are disjoint in $S$ from all components of $T_{t_+} \cap S$ and, furthermore, under the natural identification of $T_t$ with $T$, all the components of $T_{t_-} \cap S$ can be isotoped to be disjoint in $T$ from all the components of $T_{t_+} \cap S$.

Let $[a,b]$ be a maximally essential interval for $h$ relative to $S$ and let $v \in [a,b]$ be a critical value of $h|_S$. Let $t_-$ and $t_+$ be regular values on either side of $v$ such that $v$ is the unique critical value of $h|_S$ in $[t_-, t_+]$. Suppose that $\sigma_-$ is the union of the components of $T_{t_-} \cap S$ that are banded together at $v$ to produce the components $\sigma_+$ of $T_{t_+}$. The components of $\sigma_-$ are called \defn{pre-active}, those of $\sigma_+$ are called \defn{post-active}, and an arc that is pre-active or post-active is also called simply \defn{active}.  Let $\mc{A}$ be the union of all active arcs and circles and let $\mc{V}$ be the union of all the critical values $v\in[a,b]$ of $h|_S$ such that there is an active arc at $v$. Figure \ref{Fig: activearcs} shows a pre-active arc and two post-active arcs at a critical point.

\begin{figure}[ht]
\labellist \small\hair 2pt
\pinlabel {$h$} [b] at 388 188
\pinlabel {$v$} [r] at 468 116
\pinlabel {$\kappa_+$} at 104 151
\pinlabel {$\kappa'_+$} at 288 151
\pinlabel {$\kappa_-$} [t] at 221 70
\endlabellist
\centering
\includegraphics[scale=0.5]{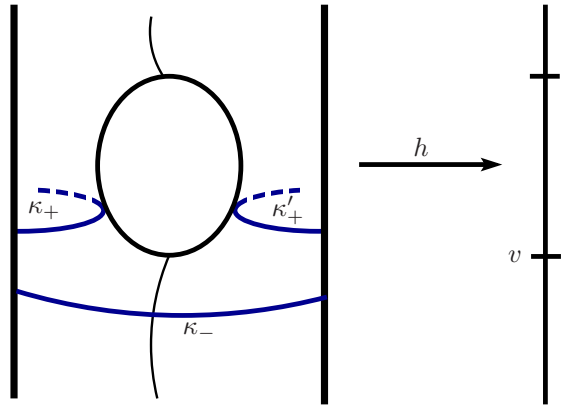}
\caption{The arc $\kappa_-$ is a pre-active arc at the critical value $v$ and the arcs $\kappa_+$ and $\kappa'_+$ are post-active arcs at $v$.}
\label{Fig: activearcs}
\end{figure}

If an arc in $\sigma_- \cup \sigma_+$ is not active, it is \defn{inactive}. Since all arcs and circles of $T_t \cap S$ are essential in both surfaces for $t \in [a,b]$, an arc $\kappa_- \subset \sigma_-$ is isotopic in $S$ to an arc $\kappa_+ \subset \sigma_+$ if and only if its projection to $T$ is isotopic in $T$ to the projection of $\kappa_+$. Let $\mc{I}$ be the union of all the inactive arcs. If $\kappa_- \subset \sigma_-$ is an inactive arc component, then there is a corresponding arc $\kappa_+$ in $\sigma_+$ such that $\kappa_-$ and $\kappa_+$ are isotopic in both $S$ and in $T_K$ (under the projection of $T_t - \inter{\eta}(K)$ with $T_K$).

Let $\mc{P}$ be the closure of the components of $S - \mc{A}$. For a component $P_k \subset \mc{P}$, let $b_k$ denote the number of copies of active arcs in $\boundary P_k$ (counted with multiplicity). Define the \defn{index} of $P_k$ to be
 \[
 J(P_k) = b_k/2 - \chi(P_k).
 \]
As in \cite[Theorem 3.1]{MHL1}, note the following:
\begin{enumerate}
\item[(a)] $J(P_k) \geq 0$ for all components $P_k \subset \mc{P}$.
\item[(b)] If a component $P_k \subset \mc{P}$ contains a critical point of $h|_S$, then $J(P_k) \geq 1$
\item[(c)] A component $P_k \subset \mc{P}$ containing a critical point of $h|_S$ has at most two post-active arcs in its boundary.
\item[(d)] $\sum_{P_k \subset \mc{P}} J(P_k) = -\chi(S)$.
\end{enumerate}

Let $Q$ be the total number of post-active arcs. By (c), $Q \leq 2|\mc{V}|$. By (a) and (b), $2|\mc{V}| \leq 2\sum_{P_k \subset \mc{P}} J(P_k)$. Hence, by (d):
\begin{equation}\label{fundamental ineq}
Q \leq -2\chi(S).
\end{equation}

Let
\[ v_1 < v_2 < \hdots < v_{m-1}\]
be the critical values of $h|_S$ in $(a, b)$ and set $v_0 = a$ and $v_m = b$. Let $q_i = (v_i + v_{i-1})/2$. A \defn{constant path} is a sequence of inactive arcs $\kappa_1, \hdots, \kappa_{m}$ with $\kappa_i \subset T_{q_i} \cap S$ and all $\kappa_i$ mutually isotopic in both $S$ and $T$.

Suppose that $(\kappa_i)$ is a constant path and identify each $T_t$ with $T$. Let $\gamma_1$ be the frontier of a regular neighborhood of $\kappa_1$ in $T$. If $\alpha$ is a circle, let $\gamma_0 = \alpha$; otherwise let $\gamma_0$ be the frontier of a regular neighborhood in $T$ of $\alpha$. If $\beta$ is a circle, let $\gamma_2 = \beta$, otherwise let $\gamma_2$ be the frontier of a regular neighborhood of $\beta$ in $T$. Note that $\gamma_0$, $\gamma_1$, and $\gamma_2$ are all essential circles in $T_K$. Recall also that the interior of $\kappa_1$ is disjoint from $\alpha$, the interior of $\kappa_m$ is disjoint from $\beta$, and $\kappa_1$ and $\kappa_m$ are isotopic in $S$. Thus, if neither endpoint of $\kappa_1$  is on the same boundary component of $T_K$ as an endpoint of either $\alpha$ or $\beta$ then $\gamma_0, \gamma_1, \gamma_2$ is a path of length 2 in $\mc{C}(T)$. See Figure \ref{Fig: path}. However, $\gamma_0 \in \mc{D}^\down_\mc{C}$ and $\gamma_2 \in \mc{D}^\up_\mc{C}$, so $d_\mc{C}(T) \leq 2$, contradicting the hypotheses of the theorem. Consequently, whenever $(\kappa_i)$ is a constant path, one endpoint of $\kappa_1$ lies on a component of $\boundary T_K$ adjacent to $\alpha$ or $\beta$.

\begin{figure}[ht]
\labellist \small\hair 2pt
\pinlabel {$1$} [r] at 2 395
\pinlabel {$2$} [br] at 14 431
\pinlabel {$3$} [br] at 41 457
\pinlabel {$4$} [b] at 76 467
\pinlabel {$\hdots$} at 74 354
\pinlabel {$\hdots$} at 319 329
\pinlabel {$\hdots$} at 562 347
\pinlabel {$\hdots$} at 562 80
\pinlabel {$\hdots$} at 74 65
\pinlabel {$\hdots$} at 319 143
\pinlabel {$|\boundary_K S|\Delta$} [tr] at 13 359
\pinlabel {$\alpha$} [br] at 209 417
\pinlabel {$\beta$} [bl] at 319 11
\pinlabel {$\kappa_1$} [r] at 263 273
\endlabellist
\centering
\includegraphics[scale=0.5]{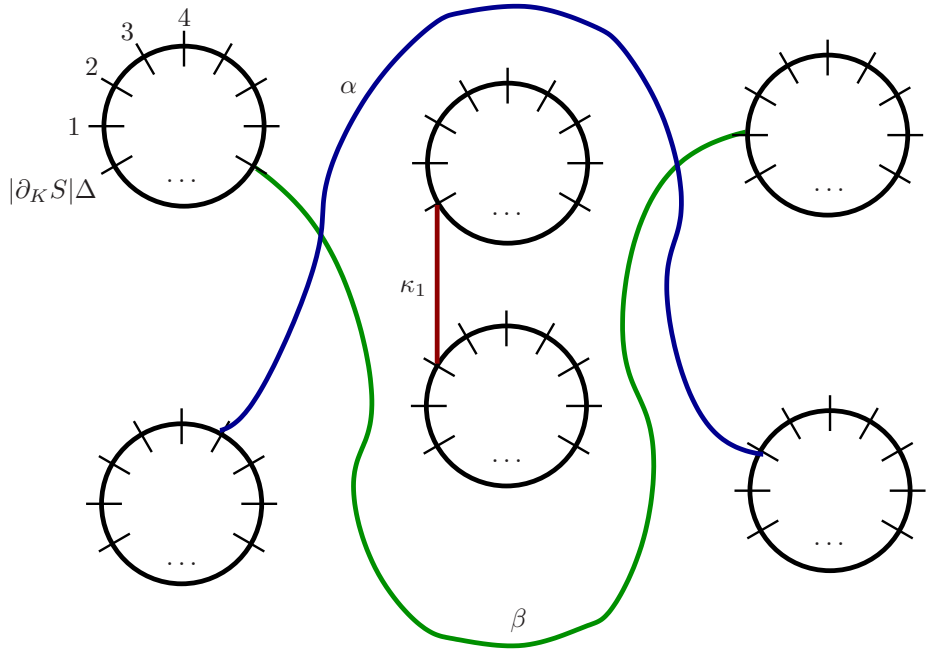}
\caption{The loops enclosing $\alpha$, $\beta$, and $\kappa_1$ and the boundary components of $T_K$ adjacent to their endpoints form a path of length 2 in $\mc{C}(T)$.} 
\label{Fig: path}
\end{figure}

Since $K$ is a knot, for any regular value $t$ of $h|_S$ and any boundary component $\sigma_t$ of $T_t - \inter{\eta}(K)$, there are exactly $|\boundary_K S|\Delta$ arcs of $T_t \cap S$ adjacent to $\sigma_t$. On $\sigma_t$, label the intersection points with $\boundary S$,
\[
1, \hdots, |\boundary_K S|\Delta
\]
with the labelling chosen so that it remains constant for all $t$. We can consider those labels to lie in a component $\sigma$ of $\boundary T_K$. Call a label in $\sigma$ \defn{active} if, for some $t \in [a,b]$, it is adjacent to an active arc and \defn{inactive} otherwise. Each inactive label corresponds to an endpoint of an arc in a constant path. Each arc in a constant path is adjacent to one of the components of $\boundary T_K$ incident to either $\alpha$ or $\beta$, so there are at most $8|\boundary_K S|\Delta$ inactive labels in $\boundary T_K$. There are $2b(T)|\boundary_K S|\Delta$ labels in $\boundary T_K$, so there are at least $(2b(T) - 8)|\boundary_K S|\Delta$ active labels. Each active arc is adjacent to two active labels. Thus, by Inequality \eqref{fundamental ineq},
\[
(b(T) - 4)|\boundary_K S|\Delta \leq Q \leq -2\chi(S).
\]
We have $-2\chi(S) = 4\genus(S) - 4|S| +2 |\boundary_0 S|+ 2|\boundary_K S|$.
Consequently,
\[
(b(T) - 4)\Delta \leq \big(4\genus(S) - 4|S| +2 |\boundary_0 S|\big)/|\boundary_K S| + 2.
\]\end{proof}

\begin{proof}[Proof of Corollary \ref{apps2}]
Let $K$ be a knot in a closed 3-manifold $M$. Let $S$ be a compact, connected, orientable essential surface of genus $g$ in $M_K$. Assume that $S$ has non-empty and non-meridional boundary. If $T$ is a bridge surface for $(M,K)$ such that $\boundary T_K \cap \boundary S$ meet minimally, then there cannot be a component of $T \cap S$ that is essential in $S$ but inessential in $T$, for then $S$ would be compressible or boundary compressible in $M_K$. Since $\boundary M_K$ is a torus, this would contradict the assumption that $S$ is essential. Thus, any component of $T_K \cap S$ that is essential in $S$ is also essential in $T$. Let $h$ be a sweepout for $M$ corresponding to $T$. Assume that $h$ has been isotoped so that $h|_S$ is Morse with critical points at distinct heights and so that $|\boundary S \cap \boundary T_t|$ is constant on $(0,1)$. When $t$ is near 1, every component of $T_t \cap S$ bounds a disc or boundary compressing disc in $S \cap T_\up$. When $t$ is near 0, every component of $T_t \cap S$ bounds a disc or boundary compressing disc in $S \cap T_\down$. Standard arguments (see, for example \cite[Corollary 3.2]{MHL1}) imply that there are regular values $a < b$ for $h|_S$ such that for every regular value $t \in [a,b]$ every component of $T_t \cap S$ is essential in $S$, and, therefore, also in $T_t$. The interval $[a,b]$ is essential for $S$ relative to $T$. We may, in fact, pick $[a,b]$ to be maximally essential. Theorem \ref{Main Theorem} implies, therefore, that if $d_\mc{C}(T) \geq 3$, then
\begin{equation}\label{bridge bound}
b(T) \leq \big(4g - 4 \big)/|\boundary S|\Delta + 2/\Delta + 4.
\end{equation}

If $S$ is planar,  then $b(T) < 6$. Since $b(T)$ is an integer, $b(T) \leq 5$.

If $S$ is non-planar, then
\[
b(T) \leq 4g - 4 + 2/\Delta + 4 \leq 4g + 2.
\]
Since $4g + 2 \geq 6$ if $g\geq 1$, we have proven our corollary.
\end{proof}

\begin{proof}[Proof of Theorem \ref{apps1}]
Assume that $K$ is non-trivial. Let $S$ be a minimal genus Seifert surface for $K$ of genus $g \geq 1$ (such a surface always exists if $K$ is null-homologous in $M$). Corollary \ref{apps2} implies $ b(T) \leq 4g + 2$. This is Conclusion (1).

If $K$ has a reducing surgery, let $\wihat{S}$ be an essential sphere in the surgered manifold that intersects the core of the surgery solid torus $\wihat{K}$ minimally. The surface $S = \wihat{S} \cap M_K$ is an essential non-meridional planar surface in $M_K$, so Corollary \ref{apps2} implies
\[
b(T) \leq 5,
\]
giving Conclusion (2).

If $K$ is atoroidal, but has a toroidal surgery, let $\wihat{S}$ be an essential torus in the surgered manifold that intersects the core of the surgery solid torus $\wihat{K}$ minimally. Let $S = \wihat{S} \cap M_K$. The surface $S$ is an essential non-meridional genus 1 surface in $M_K$. Corollary \ref{apps2} implies
\[
b(T) \leq 6.
\]
If the surgery slope is non-integral (i.e. if $\Delta \geq 2$) then inequality \eqref{bridge bound} gives the better bound of $b(T) \leq 5$.
\end{proof}


\end{document}